\documentclass[12pt,reqno]{amsart}
\usepackage{amsmath}
\usepackage{amsfonts, amssymb, euscript, mathrsfs}
\usepackage{amsthm, upref}
\usepackage{graphicx}
\usepackage[usenames, dvipsnames]{color}

\numberwithin{equation}{section}
\allowdisplaybreaks
\usepackage{hyperref}
\newcommand*{\Cdot}{{\raisebox{-0.5ex}{\scalebox{1.8}{$\cdot$}}}} 

\usepackage{tikz}
\usepackage{enumerate}
\usepackage[paper=letterpaper,margin=0.795in]{geometry}
\usepackage{aliascnt}
\usepackage{hyperref}
\usepackage{verbatim}

\renewcommand{\tilde}{\widetilde}

\newcommand{\MZ}{\mathbb{Z}}

\newcommand{\BR}{\mathbb{R}}

\newcommand{\SL}{\sum\limits}

\newcommand{\al}{\alpha}

\newcommand{\La}{\Lambda}
\newcommand{\ME}{\mathbf E}

\newcommand{\CF}{\mathcal F}

\newcommand{\CU}{\mathcal U}
\newcommand{\MP}{\mathbf P}

\newcommand{\CN}{\mathcal N}

\newcommand{\Oa}{\Omega}

\renewcommand{\phi}{\varphi}

\newcommand{\eps}{\varepsilon}
\newcommand{\la}{\lambda}

\newcommand{\ol}{\overline}

\renewcommand{\comment}[1]{}

\newcommand{\mP}{\mathbf{p}}
\newcommand{\md}{\mathrm{d}}

\DeclareMathOperator{\Var}{Var}

\DeclareMathOperator{\Exp}{Exp}

\begin{document}

\theoremstyle{plain}
\newtheorem{thm}{Theorem}[section]
\newtheorem*{thmnonumber}{Theorem}
\newtheorem{lemma}[thm]{Lemma}
\newtheorem{prop}[thm]{Proposition}
\newtheorem{cor}[thm]{Corollary}
\newtheorem{open}[thm]{Open Problem}
\newtheorem{conj}[thm]{Conjecture}

\theoremstyle{definition}
\newtheorem{defn}{Definition}
\newtheorem{asmp}{Assumption}
\newtheorem{notn}{Notation}
\newtheorem{prb}{Problem}

\theoremstyle{remark}
\newtheorem{rmk}{Remark}
\newtheorem{exm}{Example}
\newtheorem{clm}{Claim}

\author{Andrey Sarantsev and Li-Cheng Tsai}

\title[Stationary Gap Distributions]
{Stationary Gap Distributions for\\
Infinite Systems of Competing Brownian Particles}

\address[A Sarantsev]{
	Department of Statistics and Applied Probability,
	\newline\hphantom{\ \ A Sarantsev}
	University of California, Santa Barbara}
\email{sarantsev@pstat.ucsb.edu}

\address[L-C Tsai]{Department of Mathematics, Columbia University}
\email{lctsai.math@gmail.com}



\begin{abstract}
Consider the infinite Atlas model: a semi-infinite collection of particles driven by independent standard Brownian motions with zero drifts, except for the 
bottom-ranked particle which receives unit drift.
%
We derive a continuum one-parameter family of product-of-exponentials stationary gap distributions, with exponentially \emph{growing} density at infinity. 
This result shows that there are infinitely many stationary gap distributions for the Atlas model,
and hence resolves a conjecture of Pal and Pitman (2008) \cite{PP2008} in the negative.
This result is further generalized for infinite systems of competing Brownian particles with generic rank-based drifts.
\end{abstract}

\maketitle
\thispagestyle{empty}
\section{Introduction and Main Results} 
Consider a system of infinitely many Brownian particles on the real line: $X_i(t)$, $i = 1, 2, \ldots$, $t \ge 0$. Assume we can rank them from bottom upward at any time $t \ge 0$: $X_{(1)}(t) \le X_{(2)}(t) \le \ldots$, and they satisfy the following system of SDEs:
\begin{equation}
	\label{eq:infinite-Atlas}
	\md X_i(t) = 1\left(X_i(t) = X_{(1)}(t)\right)\md t + \md W_i(t),\ \ i = 1, 2, \ldots,
\end{equation}
where $W_1, W_2, \ldots$ denote independent Brownian motions.
In plain English, the bottom particle moves as a Brownian motion with drift one, and all other particles move as driftless Brownian motions. This system of Brownian particles is called the \emph{infinite Atlas model}, for the bottom particle supporting all other particles ``on its shoulders'', as the ancient Atlas hero. 

\subsection{Infinite systems of competing Brownian particles}

Although the main interest of our work is the infinite Atlas model \eqref{eq:infinite-Atlas}, our result can be naturally generalized to more general systems of competing Brownian particles. 
In this subsection, we rigorously define these infinite systems. 
Finite systems of competing Brownian particles are defined very similarly in Section~\ref{sect:finite}. 

Letting $\mathbb{Z}_{>0} := \{1, 2, \ldots\}$, $\BR_+ := [0, \infty)$,
we adopt the notations $ \BR^\infty := \{(x_1,x_2,\ldots)|x_i\in\BR\} $ and 
$ \BR^\infty_+ := \{(z_1,z_2,\ldots)|z_i\in\BR_+\} $ for infinite dimensional spaces.
We say an infinite sequence $ x = (x_i)_{i \ge 1} \in\BR^\infty $ is \emph{rankable} if there exists a ranking permutation $ p: \mathbb{Z}_{>0} \to \mathbb{Z}_{>0} $ 
such that $ x_{p(i)} \leq x_{p(j)} $, for all $ i < j \in \mathbb{Z}_{>0} $.
Not every $ x\in\BR^\infty $ is rankable;
for example, the sequence $ x:=(x_n=\frac{1}{n})_{n=1}^\infty $ is not rankable.
To ensure that such a ranking permutation is unique,
we resolve ties in lexicographic order: if $x_{p(i)} = x_{p(j)}$ for $i < j$, then $\mP_x(i) < \mP_x(j)$.
We let $ \mP_x(\Cdot): \mathbb{Z}_{>0} \to \mathbb{Z}_{>0} $
denote the unique ranking permutation for a rankable $ x $.

Hereafter,
\emph{standard Brownian motion} refers to
a one-dimensional Brownian motion with zero drift and unit diffusion coefficient.
Throughout this paper, we operate on a filtered probability space 
$(\Oa, \CF, (\CF_t)_{t \ge 0}, \MP)$ with the filtration satisfying the usual conditions, 
and fix independent standard Brownian motions $ W_1, W_2, \ldots $ 
with respect to the filtration $ (\CF_t)_{t \ge 0} $.

\begin{defn}
\label{defn:infinite-CBP}
Assume $X = (X(t), t \ge 0)$ is an $\BR^{\infty}$-valued adapted process such that  
$X(t) = (X_i(t))_{i \ge 1}$ is rankable for every $t \ge 0$, 
each coordinate $X_i = (X_i(t), t \ge 0)$ is a.s.\ continuous, and
\begin{equation}
	\label{eq:infinite-CBP}
	\md X_i(t) 
	= 
	\left[\SL_{k=1}^{\infty}1\left(\mP_{X(t)}(k) = i\right)g_k\right]\md t + \md W_i(t),\ \ i = 1, 2, \ldots
\end{equation}
Then $ X $ is called an \emph{infinite system of competing Brownian particles} 
with \emph{drift coefficients} $ g_1, g_2, \ldots$
We adopt the notation $ Y_k(t) := X_{\mP_{X(t)}(k)}(t) $ for the \emph{$k$th ranked particle}, and $ Z_k(t) := Y_{k+1}(t) - Y_{k}(t) $ for the \emph{$k$th gap}. The $\BR^{\infty}_+$-valued process $Z = (Z(t), t \ge 0)$, $Z(t) = (Z_k(t))_{k \ge 1}$, is called the \emph{gap process}.  
Each $X_i = (X_i(t), t \ge 0)$ is called the \emph{$i$th named particle}. 
Throughout this paper we consider rankable initial conditions,
and assume without lost of generality that the initial condition $ X(0) $ is \emph{standardized}.
That is, 
\begin{align*}
0 = X_1(0) \leq X_2(0) \leq X_3(0) \leq \ldots
\end{align*}
\end{defn}

A sufficient condition for the existence and uniqueness of \eqref{eq:infinite-CBP} is given by \cite{MyOwn6}.
To state this result, we define the configuration space of named particles:
\begin{align}
	\label{eq:Uspc}
	\CU 
	= 
	\Big\{ 
		x = (x_i)_{i \ge 1} \in \BR^\infty \, \Big| \,
		\lim\limits_{i \to \infty}x_i = \infty, 
		\text{ and } \SL_{i=1}^{\infty}e^{-\al x_i^2} < \infty,
		\text{ for all } \al > 0
	\Big\},
\end{align}
as well as the corresponding space of \emph{gaps}:
\begin{align}
	\label{eq:Vspc}
	\mathcal{V} := \{ (z_k)_{k=1}^\infty \in \BR^\infty_+ \, | \, (0, z_1,z_1+z_2,z_1+z_2+z_3,\ldots) \in \CU \}.
\end{align}
\begin{prop}[{\cite[Theorem~3.2]{MyOwn6}}] 
Assume that $ x \in \CU $ and the drift coefficients $(g_n)_{n \ge 1}$ satisfy 
\begin{equation}
	\label{eq:drifts}
		\SL_{k=1}^{\infty}g_k^2 < \infty.
\end{equation}
Then there exists in the weak sense a unique in law version of the infinite system~\eqref{eq:infinite-CBP} with $X(0) = x$. 
In this case, $ X(t) \in \CU $ for every $t \ge 0$ a.s.
\label{prop:existence-inf}
\end{prop}
\begin{rmk} If, instead of \eqref{eq:drifts}, 
we impose a stronger condition on the drift coefficients: 
the sequence of drifts eventually vanishes, that is,
\begin{equation}
\label{eq:drifts-stable}
\mbox{for some}\ \ n_0,\ \ g_{n_0} = g_{n_0+1} = \ldots = 0,
\end{equation}
then the system~\eqref{eq:infinite-CBP} exists in the strong sense and is pathwise unique, see \cite{IKS2013}. 
\end{rmk}

\begin{rmk} 
The gap process $ Z=Z(t) $ is invariant under adding a drift $ g_\infty dt $
to each named particle. 
Therefore, the conditions~\eqref{eq:drifts} is readily generalized to 
\begin{align*}
	\lim_{k \to \infty}g_k = g_{\infty},
	\quad \text{and }
	\sum_{k=1}^{\infty} (g_k - g_{\infty})^2 < \infty.
\end{align*}
Similarly, the condition~\eqref{eq:drifts-stable} is generalized to the condition
$
	g_{n_0} = g_{n_0+1} = \ldots = g_\infty.
$
\end{rmk}

\subsection{Main result}
\label{sect:main}
The question of current interest is to find stationary distributions for the gap process $ Z(t) $. Let us first rigorously define this concept. Take an infinite system $X$ of competing Brownian particles with drift coefficients $g_1, g_2, \ldots$; let $Z$ be its gap process.
\begin{defn} A probability measure $\pi$ on $\BR^{\infty}_+$ is called a \emph{stationary gap distribution} 
or a \emph{quasi-stationary distribution} for the system $ X $ 
if there exists in the weak sense a unique in law version of \eqref{eq:infinite-CBP}  with $Z(0) \sim \pi$, 
and for this version we have: $Z(t) \sim \pi$ for every $t \geq 0$. 
\end{defn}

Let $\Exp(\la)$ denote the exponential distribution with mean $\la^{-1}$,
i.e.\ having density $\la e^{-\la x}\md x, x > 0$.
The following stationary distribution of the gap process of the Atlas model \eqref{eq:infinite-Atlas}
was derived by Pal and Pitman~\cite{PP2008}:
\begin{equation}
	\label{eq:pi-infty}
	\ol{\pi} := \bigotimes\limits_{k=1}^{\infty}\Exp(2).
\end{equation}
Samples from this distribution 
are configurations of particles on $ \BR_+ $ of roughly uniform density $ 2 $,
where the value $ 2 $ arises from 
the balancing between the unit drift $ g_1=1 $ and the push-back from the crowd of particles,
as heuristically explained in \cite{Aldous}.
It was further shown in \cite{DemboTsai} that,
under \eqref{eq:pi-infty}, each ranked particle
$ Y_k(t) $ typically deviates $ O(t^{1/4}) $ from its starting location $ Y_k(0) $ for large $ t $.

Here, we provide a one-parameter family of stationary gap distributions $ \pi_a $,
with drastically distinct behaviors:
the density grows exponentially as $ x\to\infty $
and each rank particle $ Y_k $ travels linearly in time (in expectation).
Denote the average of the first $ n $ drift coefficients by $ \ol{g}_n $:
\begin{align}
	\label{eq:olg}
	\ol{g}_n := \tfrac1n \left(g_1 + \ldots + g_n \right).
\end{align}
The following is the main result of this paper.
\begin{thm}
\label{thm:main}
Consider an infinite system of competing Brownian particles from~\eqref{eq:infinite-CBP} with drift coefficients satisfying~\eqref{eq:drifts}. Take any real number $a$ such that 
\begin{align}
	\label{eq:a:cnd}
	a > -2\inf\limits_{n \ge 1}\ol{g}_n.
\end{align}
(a) The following measure $ \pi_a $ is supported on $ \mathcal{V} $,
and is a stationary distribution for the gap process:
\begin{align}
	\pi_a := \bigotimes\limits_{n=1}^{\infty}\Exp\left(2(g_1 + \ldots + g_n) + na\right).
	\label{eq:pi-a-general}
\end{align}
(b) If $Z(0) \sim \pi_a$: the system is in this stationary distribution, then 
\begin{align*}
	\ME\left(Y_k(t) - Y_k(0)\right) = -\tfrac a2t,\ \ t \ge 0,\ \ k = 1, 2, \ldots
\end{align*}
\end{thm}

We now provide some important special cases of the general systems considered in Theorem~\ref{thm:main}.
\begin{exm} 
Infinite Atlas model: $g_1 = 1$, and $g_k = 0$ for $k \ge 2$. 
Then $\inf_{n \ge 1}\ol{g}_n = 0$, so for $a > 0$,
we have the following family of stationary distributions:
\begin{align}
	\label{eq:pa:Atlas}
	\pi_a := \bigotimes\limits_{n=1}^{\infty}\Exp(2 + na).
\end{align}
\end{exm}

\begin{exm} 
\label{exm:indep}
Independent Brownian motions: $g_1 = g_2 = \ldots = 0$, so $\inf_{n \ge 1}\ol{g}_n = 0$, and for $a > 0$ we have the following family of stationary distributions:
$$
\pi_a := \bigotimes\limits_{n=1}^{\infty}\Exp(na).
$$
\end{exm}

\begin{exm}
\label{exm:iAtlas}
The ``inverted Atlas'' model, where the bottom particle has negative drift: $g_1 = -1$, $g_2 = g_3 = \ldots = 0$. Then $\inf_{n \ge 1}\ol{g}_n = -1$, and for $a > 2$ we get:
$$
\pi_a := \bigotimes\limits_{n=1}^{\infty}\Exp(-2 + na).
$$
\end{exm}

\begin{rmk}\label{rmk:drfts}
Actually, the condition~\eqref{eq:drifts} does not play a crucial role in the proof of Theorem~\ref{thm:main}.
More precisely,
under the weaker condition $ \sup|g_n| < \infty $, 
our proof of Theorem~\ref{thm:main} still
applies for constructing a copy of the infinite system with $ Z(t) \sim \pi_a $ for all $t \ge 0$.
The stronger condition~\eqref{eq:drifts} is assumed merely 
to ensure that the solution to \eqref{eq:infinite-CBP} is unique in law,
so that the notion of stationary gap distribution is well-defined.
\end{rmk}

Theorem~\ref{thm:main} shows that
the stationary gap distributions for systems of competing Brownian particles
(and in particular for the infinite Atlas model) are \emph{not} unique.
In fact, as we further show in Appendix~\ref{sect:appd},
the distributions $ \pi_a $ are mutually \emph{singular} for different values of $a$. 
This result in particular resolves the conjecture \cite[Conjecture~2]{PP2008}
of Pal and Pitman in the negative.
As mentioned previously,
for any $ a $ satisfying \eqref{eq:a:cnd},
the distribution $ \pi_a $ exhibits exponentially growing density as $ x\to\infty $.
To see why this is true,
assuming the condition \eqref{eq:drifts-stable} for simplicity,
for $ (\zeta_k)_{k=1}^\infty \sim \pi_a $,
we note that 
\begin{align*}
	L_n := \sum_{k=1}^n \ME(\zeta_k) = \sum_{k=1}^n \frac{1}{g_1+\ldots+g_{k}+ka} = a^{-1} \log n +c_n,	
\end{align*}
where $ \{c_n\} $ is a bounded sequence.
Inverting this relation yields $ n = c'_n e^{aL_n} $, where $ c'_n := e^{-ac_n} $.
This suggests that there are typically (up to a proportion) $ e^{aL} $ particles within an interval $ [0,L] $.
A precise statement of this is given and proven in Appendix~\ref{sect:appd}.

For the discrete-time analogue of independent Brownian particles from Example~\ref{exm:indep}, quasi-stationary distributions of the type $ \pi_a $
already appeared in the study of the Sherrington--Kirkpatrick model of spin glasses \cite{RA}. Such a distribution arises naturally for independent Brownian particles.
However, it is far from obvious that similar quasi-stationary distributions should appear 
in the context of competing Brownian particles,
since rank-based drifts introduce complicated dependence among particles.

Rather, the product-of-exponential distribution $ \pi_a $ arises from the study of 
Reflected Brownian Motion (RBM).
We give a heuristic derivation of the distribution $ \pi_a $
using RBM in the infinite-dimensional positive orthant $\BR^{\infty}_+$ in Section~\ref{sect:RBM}.
To justify this heuristic derivation (i.e.\ to prove Theorem~\ref{thm:main})
requires taking a sequence of finite systems of competing Brownian particles 
with suitable drift coefficients $ (g_{k,N})_{k=1}^N $
and showing that the sequence converges to the infinite system.
Even for the Atlas model,
where $ g_1 =1 $ and $ g_2=g_3=\ldots=0 $,
we need to construct $ g_{k,N} $ that \emph{varies} in a suitable way over $ k=2,\ldots,N $,
in order to simulate the pressure caused by the exponentially dense particles at $ x \to \infty $;
see \eqref{eq:drift-approx}.
This is in sharp contrast 
with the derivation of the measure $ \ol{\pi} $ \eqref{eq:pi-infty},
where $ (g_{k,N})_{k=1}^N $ can be taken to be $ (1,0,\ldots,0) $.

Theorem~\ref{thm:main} further demonstrates
a sharp contrast between finite and infinite systems of competing Brownian particles,
regarding the criteria for having stationary gap distributions.
For a finite system to have a stationary gap distribution,
the \emph{stability condition}
\begin{equation}
	\label{eq:stability}
	\ol{g}_k > \ol{g}_N,\ \ k = 1, \ldots, N - 1
\end{equation}
 must hold (see Proposition~\ref{prop:stability}),
as \eqref{eq:stability} imposes a ``crowding'' mechanism on the rank particles.
On the other hand, for an infinite systems,
the stationary gap distribution $ \pi_a $ may exist even without
any form of crowding mechanisms from the drifts.
As we see in Example~\ref{exm:indep}, the drifts are not in effect.
In Example~\ref{exm:iAtlas}, the drifts introduce a ``repelling'' mechanism---the \emph{opposite} of a crowding mechanism.
The sharp contrast between finite and infinite systems
is due to the additional crowding effect, in infinite systems,
caused by pressure from exponentially growing density under $ \pi_a $.

\subsection{Conjectures} 
Here we state some conjectures related to Theorem~\ref{thm:main}.
First we recall that, 
for more general systems of competing Brownian particles than the Atlas model,
\cite{MyOwn6} derived the following stationary gap distribution
\begin{align}
	\label{eq:pi-infty:}
	\pi_0 := \bigotimes\limits_{k=1}^{\infty}\Exp\left(2(g_1 + \ldots + g_k )\right).
\end{align}
This is done in \cite[Section~4.2]{MyOwn6} under the condition~\eqref{eq:drifts} and an additional condition
that there exists $ N_1<N_2<\ldots\to\infty $ such that
\begin{align}
	\label{eq:drifts:Nj}
	\ol{g}_{k} > \ol{g}_{N_j},
	\text{ for }
	k=1,\ldots, N_{j}-1,\ \ j \ge 1.
\end{align}
\begin{rmk}
It follows from Theorem~\ref{thm:main} that $ \pi_0 $ is supported on $ \mathcal{V} $.
To see this, fix a positive $ a>0 $ satisfying the condition~\eqref{eq:a:cnd}.
Let $ \zeta = (\zeta_n)_{n=1}^\infty \sim \pi_0 $ and $ \zeta'=(\zeta'_n)_{n=1}^\infty\sim\pi_a $
be gap processes sampled from the designated distributions.
With $ a>0 $, 
comparing \eqref{eq:pi-infty:} and \eqref{eq:pi-a-general},
we find that $ \zeta $ stochastically dominates $ \zeta' $.
That is, there exists a coupling of $ \zeta,\zeta' $ under which
\begin{align}
	\label{eq:zeta:cmp}
	\zeta_n \geq \zeta'_n,	\ \text{for all } n =1,2,\ldots,	\quad \text{a.s.} 
\end{align}
By Theorem~\ref{thm:main}(a),
we have $ \zeta' \in \mathcal{V} $ a.s.
Combining this with \eqref{eq:zeta:cmp} yields $ \zeta\in\mathcal{V} $ a.s.
\end{rmk}

This stationary gap distribution~\eqref{eq:pi-infty:} 
generalizes the distribution~\eqref{eq:pi-infty} for the Altas model.
Here we use the notation $ \pi_0 $ to unify notation with \eqref{eq:pi-a-general}.
Note that under the conditions~\eqref{eq:drifts} and \eqref{eq:drifts:Nj},
we necessarily have $ \inf_n \ol{g}_n =0 $.
With this, under the preceding notations, $ \pi_a $ 
is a stationary gap distribution for all $ a\in[0,\infty)=\BR_+ $,
\emph{including} $ a=0 $.
We now conjecture that, the mixtures of these measures, over different values of $ a\in\BR_+ $,
exhaust all stationary gap distributions:
\begin{conj}
\label{conj:invt}
Under the conditions~\eqref{eq:drifts} and \eqref{eq:drifts:Nj},
any stationary gap distribution of an infinite system of competing Brownian particles 
is of the following form, for some probability measure $\rho$ on $\BR_+$:
\begin{align*}
	\pi_{\rho}(\Cdot) := \int_{\BR_+}\pi_a(\Cdot) \rho(\md a).
\end{align*} 
\end{conj}
\begin{rmk}
For the discrete time analog of the driftless system (i.e.\ $ g_1=g_2=\ldots =0$),
\cite{RA} has already proven the analogous statement as in Conjecture~\ref{conj:invt}.
Driftless systems differ from the systems considered in Conjecture~\ref{conj:invt}
in that the former does not satisfies the condition~\eqref{eq:drifts:Nj}.
Consequently, driftless systems lack stationary gap distribution of the type $ \pi_0 $,
and the statement in \cite{RA} involves only the parameter $ a>0 $.
%
\end{rmk}

A natural open problem following Theorem~\ref{thm:main}
is the large time behavior of each rank particle $ Y_k(t) $.
In view of Theorem~\ref{thm:main}(b), here we conjecture:
\begin{conj}
Fix $ (g_n)_{n \geq 1} $ and the parameter $ a $ as in Theorem~\ref{thm:main}.
Initiating the system of competing Brownian particles
at the configuration $ X_1(0)=0 $ and $ (Z_k(0))_{k=1}^\infty \sim \pi_a $,
we have that, for any fixed $ k \in \MZ_{>0} $,
\begin{align*}
	\frac{ Y_k(t) }t  \to -\frac{a}{2} \ \ \text{ as }  t \to \infty\, \text{ a.s.}
\end{align*}
\end{conj}

\subsection{Motivation and literature review} 
The Atlas model and the more general systems of competing Brownian particles
are models of interest in mathematical finance.
In particular, finite systems of competing Brownian particles
(with rank-based drifts and rank-based diffusion coefficients) 
were introduced in \cite{BFK} for the purposes of stock market modeling.
Weak existence and uniqueness in law for these systems 
follows from the earlier work of \cite{BassPardoux}.
Specific applications to mathematical finance  
include the study of: stability of the capital distribution \cite{CP2010}, 
market models with splits and mergers \cite{MyOwn4},
and portfolio optimization in \cite{JR2013b}.
Furthermore, 
finite systems of competing Brownian are of interest in their 
due to their intruding mathematical features.
There has been extensive study on
various aspects of their properties, including:
deriving the unique stationary gap distribution \cite{PP2008, 5people};
weak convergence to this stationary distribution \cite{IPS2013, MyOwn10};
the stochastic monotonicity \cite{MyOwn2};
small noise limits \cite{JR2014};
propagation of chaos \cite{JM2008};
refined properties of two dimensional systems \cite{FIKP2013};
and the question of triple collision 
(when three or more particles occupy the same position at the same time)
\cite{IK2010, IKS2013, MyOwn5, MyOwn3}. 
The last question is important
because the strong solution of a finite system of competing Brownian particles 
is only proved to exist until the first triple collision, \cite{IKS2013}. 
 
In addition to their role in mathematical finance, 
systems of competing Brownian particles 
arise as the continuum limit of exclusion processes \cite{KPS2012},
and also serve as a discrete analogue 
of a nonlinear diffusion governed by a McKean-Vlasov stochastic differential equation. 
In fact, a nonlinear diffusion can be approximated by finite systems competing Brownian particles, see \cite{S2012, JR2013a, Reygner, 4people}. 

\emph{Infinite} systems arise as natural models of large systems.
Specifically,
infinite systems of competing Brownian particles 
were first introduced in \cite{PP2008} for a special case of the infinite Atlas model, 
and later in \cite{S2011, IKS2013} for the general case,
as well as in \cite{MyOwn11} for two-sided systems $ X=(X_n)_{n \in \MZ} $.
Existence and uniqueness were established in \cite{S2011, IKS2013, MyOwn6}. 
As mentioned previously,
these infinite models exhibit stationary gap distributions $\pi_0$ from~\eqref{eq:pi-infty:} (in particular, $ \ol{\pi} $ from~\eqref{eq:pi-infty} for the infinite Atlas model)
of the desired product-of-exponential form.
This is shown in \cite{PP2008} for the infinite Atlas model and in \cite{MyOwn6} for general systems. In the latter paper \cite{MyOwn6}, the question of weak convergence of $Z(t)$ as $ t \to \infty $ was also studied.
As models of large systems,
the infinite Atlas model is naturally related to 
a certain stochastic partial differential equation \cite{DemboTsai}. 
Furthermore, as mentioned previously,
the driftless system already appeared in the description of the infinite volume limit
of the Sherrington--Kirkpatrick model.
See \cite{RA, AA, S2011} and the references therein.

There are several generalizations of these models: systems of competing L\`{e}vy particles, \cite{S2011, MyOwn12}; 
competing Brownian particles on the positive half-line, \cite{S2011, IKP2013} (in the former paper, these are called \emph{regulated systems}); 
competing Brownian particles with elastic collisions, \cite{Elastic, FIKP2013}; 
the case of asymmetric collisions, when particles behave after collision as if they had different mass, \cite{KPS2012}; 
second-order models, where drift and diffusion coefficients depend on both name and rank of the particle, \cite{5people, 2order}. 

\subsection{A heuristic derivation of $ \pi_a $}
\label{sect:RBM}
Here we give a heuristic derivation of the measure $ \pi_a $,
explaining how it arises from the theory of Reflected Brownian Motion (RBM). 
We shall not give detailed definition of an RBM here, 
and instead refer the readers to the classical survey \cite{Williams}.
Recall from \cite{MyOwn6} that,
under conditions of Proposition~\ref{prop:existence-inf}, the system $ Y = (Y_k)_{k \ge 1}$ of ranked particles solves 
the following infinite system of SDEs:
\begin{align}
	\label{eq:infinite-CBP:rk}
	\md Y_k(t) 
	= 
	g_k\md t + \md B_k(t) + \tfrac12\md L_{(k-1, k)}(t) - \tfrac12\md L_{(k, k+1)}(t),
	\ \ k = 1, 2, \ldots
\end{align}
Here, $ L_{(k,k+1)} $ denotes the local time at zero of $ Z_k = Y_{k+1}-Y_k $,
we let $ L_{(0, 1)}:=0 $ for consistency of notations, and 
\begin{align*}
	B_k(t) := \SL_{i=1}^{\infty}\int_0^t1\left(\mP_{X(t)}(i) = k\right)\md W_i(t),\ \ k = 1, 2, \ldots
\end{align*}
are independent standard Brownian motions.
With \eqref{eq:infinite-CBP:rk}, 
the process $ Z $ evolves as an RBM in the infinite-dimensional positive orthant $ \BR^\infty_+ $:
\begin{align}
	\label{eq:Z:RBM}
	\md Z(t) 
	= 
	g \md t + \md \tilde{B}(t) + R \md L(t),
\end{align}
where $ g := (g_k)_{k=1}^\infty $, $ \tilde{B}(t) := (B_{k+1}(t)-B_{k}(t))_{k=1}^\infty $,
$ L(t) := (L_{(k,k+1)}(t))_{k=1}^\infty $,
and $ R $ is the reflection matrix a tridiagonal matrix given by
\begin{align*}
	R = \begin{pmatrix}
		1	&-\tfrac12	&	0	&		0	&\ldots	\\
		-\tfrac12 &1	&-\tfrac12	&		0	&\ldots	\\
		0&	-\tfrac12 &1	&-\tfrac12			&\ldots	\\
		0&0&-\tfrac12 &1		&\ldots	\\
		\vdots&\vdots&\vdots&\vdots	& \ddots
	\end{pmatrix}.
\end{align*}
For \emph{finite}-dimensional RBM in the orthant, a sufficient condition for 
having product-of-exponential stationary distributions is
the skew-symmetry condition (see, e.g.\ \cite[Proposition~2.1]{MyOwn6} or \cite{Williams}).
It is straightforward to verify that
finite dimensional truncations of \eqref{eq:Z:RBM}
(i.e.\ \eqref{eq:ranked-SDE} in the following)
satisfy the skew-symmetry condition, and have the stationary distribution given by
\begin{align}
	\label{eq:Rmu}
	\bigotimes_{k=1}^{N-1} \Exp(\lambda_k),
	\quad
	\lambda  := R^{-1} \mu,
\end{align}
where $ \lambda := (\lambda_{k})_{k=1}^{N-1} $ and $ \mu := (g_1-g_2,\ldots, g_{N-1}-g_N) $.

Now, even though \eqref{eq:Rmu} holds only in the finite-dimensional setting,
let us \emph{informally} adopt it for deriving stationary distributions in the infinite-dimensional setting. 
Rewrite \eqref{eq:Rmu} as $ R\lambda = \mu $
(as it is not clear that $ R^{-1} $ is well-defined in infinite dimensions).
A solution of this equation is
\begin{align}
	\label{eq:lambda*}
	\lambda^* = (\lambda^*_k)_{k=1}^\infty,
	\quad
	\lambda^*_k := 2(g_1+g_2+\ldots+g_k ),
\end{align}
which gives rise to the measure $ \pi_0 $ in \eqref{eq:pi-infty:}.
This solution, however, is not unique: solving for the null vector $ R\eta = 0 $, we have
\begin{align*}
	&\eta_1 - \tfrac12 \eta_2 =0,
\\ 
	& \tfrac12 \eta_{k-1} - \eta_{k} + \tfrac12 \eta_{k+1} =0,
	\quad
	k=2,3,\ldots,
\end{align*}
which yields $ \eta = (1,2,3,\ldots) $.
With this, we have the following general solution to~\eqref{eq:Rmu}:
\begin{align}
	\label{eq:genSln}
	\lambda := \lambda^* + a \eta,
	\quad
	\text{i.e. }
	\lambda_k := 2(g_1+g_2+\ldots+g_k ) + ka,
\end{align}
with the extra condition~\eqref{eq:a:cnd} on $ a $ to ensure that each component of 
$ \lambda $ is positive.
The solution~\eqref{eq:genSln} then suggests that 
$ \pi_a $ should be also be a stationary distribution of $ Z $.

\subsection{Organization}
In Section~\ref{sect:finite}, we introduce 
finite systems of competing Brownian particles
together with the necessary tools,
and define the finite systems that will be used to prove Theorem~\ref{thm:main}.
In Section~\ref{sect:pfmain}, we prove Theorem~\ref{thm:main} by establishing 
the convergence of the finite systems to the corresponding infinite system.
Appendix~\ref{sect:appd} is devoted to
establishing properties of the measure $ \pi_a $ 
mentioned in Section~\ref{sect:main}.
\subsection{Acknowledgements}
We would like to thank \textsc{Michael Aizenman}, \textsc{Ramon van Handel},
\textsc{Tomoyuki Ichiba}, and \textsc{Mykhaylo Shkolnikov} for help and useful discussion. 

Andrey Sarantsev was partially supported by NSF through grants
DMS~1007563, DMS~1308340, DMS~1409434, and DMS~1405210.
Li-Cheng Tsai was partially supported by the NSF through DMS~1106627 
and the KITP graduate fellowship through NSF grant PHY11-25915.

\section{Finite Systems of Competing Brownian Particles}
\label{sect:finite}
To define a finite system of competing Brownian particles,
we fix $ N \geq 2 $ to be the number of particles,
and let $ g_1,\ldots, g_N $ denote the drift coefficients.
Here $ \mP_x(\Cdot) : \{1,\ldots,N\} \to \{1,\ldots, N\} $
denote the analogous ranking permutation for $ x\in\BR^N $,
which is unique by resolving ties in the lexicographic order.
Note that unlike in infinite dimensions, \emph{any} $ x\in\BR^N $ is rankable.

\begin{defn}
Take an $ \BR^N $-valued continuous adapted process 
$$
X = (X(t), t \ge 0),\ \ X(t) = (X_1(t), \ldots, X_N(t)),\ \ t \ge 0,
$$
which satisfies the following SDEs: for $i = 1, \ldots, N$, 
\begin{equation}
	\label{eq:finite-CBP}
	\md X_i(t) = \left[\SL_{k=1}^N1\left(\mP_{X(t)}(k) = i\right)g_k\right]\md t + \md W_i(t),
	\quad
	i=1,\ldots,N.
\end{equation}
Then $X$ is called a \emph{finite system of competing Brownian particles}. Each $X_i = (X_i(t), t \ge 0)$ is called the \emph{$i$th named particle}. As in Definition~\ref{defn:infinite-CBP},
we assume without loss of generality that the initial condition $ X(0) $ is \emph{standardized:}
$ 0=X_1(0) \leq X_2(0) \leq \ldots \leq X_N(0) $.
We similarly define \emph{ranked particles} $ Y_k = (Y_k(t), t \ge 0)$, and the \emph{gap process} $Z = (Z(t), t \ge 0)$, $ Z(t) = (Z_1(t), \ldots, Z_{N-1}(t)) \in\BR^{N-1}_+ $ as
\begin{align*}
	Y_k(t) &= X_{\mP_{X(t)}(k)}(t),
	\quad
	k=1,\ldots,N,
\\
\ \ Z_k(t) &= Y_{k+1}(t) - Y_k(t),
	\quad
	k=1,\ldots, N-1.
\end{align*}
\end{defn}
\noindent
Systems with rank-based diffusion coefficients may also be considered, 
but for our purposes it is sufficient to consider unit diffusion coefficients. 
It is known from \cite{IKS2013} that,
for any deterministic initial condition $ x\in\BR^N $,
the equation~\eqref{eq:finite-CBP} always has a strong solution, 
which is pathwise unique.

In the sequel we will also need to consider the dynamics for the ranked particles $ Y_k $.
To this end, we let $ L_{(k, k+1)} = (L_{(k, k+1)}(t), t \ge 0) $ be the local time process at zero of $Z_k$, for $k = 1, \ldots, N-1$, 
and call $L_{(k, k+1)}$ the 
\emph{local time of collision between the ranked particles $Y_k$ and $Y_{k+1}$}. For consistency of notation, we let $L_{(0, 1)}(t) \equiv 0$ and $L_{(N, N+1)}(t) \equiv 0$. It was shown in \cite{BG2008, 5people} that the dynamics of ranked particles is given by 
\begin{equation}
\label{eq:ranked-SDE}
\md Y_k(t) = g_k\md t + \md B_k(t) + \tfrac12\md L_{(k-1, k)}(t) - \tfrac12\md L_{(k, k+1)}(t),\ \ k = 1, \ldots, N,
\end{equation}
where the following processes are i.i.d. standard Brownian motions:
\begin{align}
	\label{eq:Bk}
	B_k(t) := 
	\sum_{i=1}^N \int_0^t1\left(\mP_{X(s)}(k) = i\right)\md W_i(s),\ \ k = 1, \ldots, N.
\end{align}

Our strategy of proving Theorem~\ref{thm:main} is to approximate 
the infinite system \eqref{eq:infinite-CBP} by certain finite systems.
To this end, let us recall the following result
(proved in \cite{PP2008, 5people, MyOwn6})
on the necessary and sufficient condition for the existence 
of stationary gap distributions for finite systems.

\begin{prop}
Recall the notation $ \ol{g}_k $ from \eqref{eq:olg}.
There exists a stationary distribution for the gap process if and only if
the stability condition~\eqref{eq:stability}.
In this case, this stationary distribution is unique and is given by
\begin{equation}
\label{eq:finite-product}
\pi = \bigotimes\limits_{k=1}^{N-1}\Exp\left(2(g_1 + \ldots + g_k - k\ol{g}_N)\right) = \bigotimes\limits_{k=1}^{N-1}\Exp\left(2k\left(\ol{g}_k - \ol{g}_N\right)\right).
\end{equation}
In addition, if the system is initiated from this stationary distribution, that is, $Z(0) \sim \pi$, then 
\begin{equation}
\label{eq:drift-together}
\ME\left(Y_k(t) - Y_k(0)\right) = \ol{g}_Nt,\ \ k = 1, \ldots, N,\ \ t \ge 0.
\end{equation}
\label{prop:stability}
\end{prop}

Now, let us define the finite systems that will be used in the proof of Theorem~\ref{thm:main}.
For every $ m \ge 2 $, 
we let $ X^{(m)} = (X^{(m)}_k)_{k=1}^{m^2} $ be a system of $ m^2 $ competing Brownian particles:
\begin{align}
	\label{eq:Xm}
	\md X^{(m)}_i(t) = \Big[ \SL_{i=1}^{m^2} 1\left(\mP_{X^{(m)}(t)}(k) = i\right)g^{(m)}_k \Big]\md t 
	+ \md W_i(t),
	\quad
	i=1,\ldots ,m^2,
\end{align}
with the following drift coefficients:
\begin{align}
	\label{eq:drift-approx}
	g_k^{(m)} &:= 
	\begin{cases}
		g_k,\ k = 1, \ldots, m;
		\\
		b_m,\ k = m + 1, \ldots, m^2,
	\end{cases}
\\
	\label{eq:b-m}
	\text{where } 	b_m &:= -\frac{m^2}{2(m^2 - m)}a - \frac{g_1 + \ldots + g_m}{m^2 - m}.
\end{align}	
This specific choice of $ b_m $ ensures that
$
	\ol{g}^{(m)} := \frac1{m^2} ( g_1^{(m)} + \ldots + g_{m^2}^{(m)} ) = - \frac{a}{2}.
$
Letting 
\begin{align}
	\label{eq:lambdam}
	\lambda^{(m)}_k := 2( g^{(m)}_1 +\ldots+ g^{(m)}_k - k \ol{g}^{(m)}) ,\ \ k = 1, \ldots, m^2-1,
\end{align}
after elementary calculations we get:
\begin{align}
	\label{eq:lambda:1m}
\lambda^{(m)}_k &= \la_k = 2(g_1 + \ldots + g_k) + ak,& &\text{for}\ \ k = 1,\ldots,m,
\\
	\label{eq:lambda:>m}
	\lambda^{(m)}_k &= \tfrac{m^2-k}{m-1}(2\ol{g}_m+a),& &\text{for}\ \ k = m+1,\ldots,m^2-1.
\end{align}
The assumption~\eqref{eq:a:cnd} ensures that $ \lambda^{(m)}_k >0 $, for $ m=1,\ldots,m^2-1 $.
This, by \eqref{eq:lambdam}, is equivalent to $ \ol{g}^{(m)}_k > \ol{g}^{(m)}_{m^2} $,
so by Proposition~\ref{prop:stability},
$ X^{(m)} $ has the following stationary gap distribution:
\begin{equation}
	\label{eq:coincide}
	\pi^{(m)}_a :=
	\bigotimes\limits_{k=1}^{m^2 - 1}\Exp (\la_k^{(m)} ).
\end{equation}

\section{Proof of Theorem~\ref{thm:main}}
\label{sect:pfmain}

For a dimension $ d \ge 1 $ and a $ T\in\BR_+ $, 
let $C([0, T], \BR^d)$ be the space of continuous functions $ [0, T] \to \BR^d $,
and for $d = 1$, we simply write $C[0, T]$.
Hereafter, we endow this space with the standard uniform topology.
Let $ Y^{(m)} = (Y^{(m)}_k)_{k=1}^{m^2} $ and $ Z^{(m)} = (Z^{(m)}_k)_{k=1}^{m^2-1} $
denote the corresponding ranked particles  and the gap process for the system $ X^{(m)} $.
We initiate $ X^{(m)} $ at the stationary gap distribution $ \pi^{(m)}_a $, \eqref{eq:coincide}.
That is, we let
\begin{align*}
	X^{(m)}_1(0) := 0 \leq X^{(m)}_2(0) \leq X^{(m)}_3(0) \leq \ldots;\ \ \mbox{and}\ \ 
	(Z^{(m)}_k(0))_{k=1}^{m^2-1} \sim \pi^{(m)}_a.
\end{align*}

\subsection{Proof of Part~(a)}
{\it{}Step 1.}
Recall the definition of $ \mathcal{V} $ from \eqref{eq:Vspc}.
Let us first prove that the probability distribution $\pi_a$ is supported on $ \mathcal{V} $.
Indeed, denote $ g_* := \sup_{n \ge 1}|g_n|<\infty $ and let $b := 2g_* + a$. 
Then $b > 0$ by \eqref{eq:a:cnd}, and $\la_n := 2(g_1 + \ldots + g_n) + na \le bn$. Therefore, $\la^{-1}_n \ge b^{-1}n^{-1}$. 
For some bounded sequence $(c_n)_{n \ge 1}$ of real numbers,  we get:
$$
\La_n := \SL_{k=1}^n\la_k^{-1} \ge b^{-1}\SL_{k=1}^nk^{-1} = b^{-1}\log n + c_n.
$$
Applying the inequality $(a_1 + a_2)^2 \ge a_1^2/2 - a_2^2$ for all real $a_1, a_2$, we have:  
$$
\La_n^2 \ge \frac1{2b^2}\log^2n - c^2_n \ge b'\log^2n - c'\ \ \mbox{for some constants}\ \ b', c' > 0.
$$
Thus, we have:
\begin{align*}
	\SL_{n=1}^{\infty}e^{-\al \La_n^2} 
	\le 
	\SL_{n=1}^{\infty}\exp\left(-\al b'\log^2n + \al c'\right)< \infty, \text{ for all } \al > 0.
\end{align*}
Applying  \cite[Lemma 4.5]{MyOwn6}, we complete the proof that the distribution $\pi_a$ is supported on $\mathcal V$. 

\medskip

{\it{}Step 2.} For $ n' \geq n $,
we let $ [x]_{\downarrow n}: (x_1,\ldots,x_{n'}) \mapsto (x_1,\ldots,x_n) $ 
denote the projection onto the first $ n $ coordinates.
Fixing arbitrary $ n $ and $ T\in\BR_+ $,
our goal is to show that 
$ [X^{(m)}]_{\downarrow n} $
converges to a limit process $ [ X ]_{\downarrow n} $ as $ m\to\infty $,
such that $ X $ solves \eqref{eq:finite-CBP} and has a stationary gap distribution given by $ \pi_a $.
Toward this end, 
we will need to truncate the large system $ (X^{(m)}_k)_{k=1}^{m^2} $ at some fixed dimension.
This is done with the help of the following lemma.
Hereafter, to simplify notation,
we use the letter $ c $ for \emph{any generic}
positive constant that 
depends only on $ g_1,g_2,\ldots $, $ a $ and $ T $.
Slightly abusing notation,
we use \emph{the same} letter $ c $ even 
if there are multiple such constants within the same formula.

\begin{lemma}\label{lem:Xmtail}
Fix any $ T\in\BR_+ $.
There exists $ c\in(0,\infty) $ 
(depending only on $a, T, g_n,\ n \geq 1$,
as mentioned previously), such that:
\begin{align}
	\label{eq:Xuptail}
	\MP\Big( \sup_{0\leq t \leq T} X^{(m)}_k(t) \geq u \Big) 
	&\leq 
	c e^{ c(ck-u) },
	\quad
	\text{for } k=1,\ldots, m,
	\ u\in\BR,
\\
	\label{eq:Xlwtail}
	\MP\Big( \inf_{0\leq t \leq T} X^{(m)}_k(t) \leq u \Big) 
	&\leq 
	ce^{-c(\log k-u )_+^2}+
	{ck^{-2}e^{cu} },
	\quad
	\text{for } k =1,\ldots, m^2,
	\ u\in\BR.
\end{align}
\end{lemma}
\begin{rmk}
The following proof actually applies even if the term $ k^{-2} $ in \eqref{eq:Xlwtail}
is replaced by $ k^{-\ell} $, for arbitrarily large $ \ell $,
but doing so makes various constants depend also on $ \ell $.
Here we prove \eqref{eq:Xlwtail} only for $ \ell=2 $ as it suffices for our purpose.
\end{rmk}
\begin{proof}
Throughout this proof,
for $ \BR $-valued random variables $ X,Y $,
the notation $ X \succeq Y $ means that $ X $ stochastically dominates $ Y $,
and likewise for $ X \preceq Y $.
Define the standard Gaussian density and the tail distribution function: 
$$
 \psi(y) := \frac{1}{\sqrt{2\pi}} e^{-y^2/2} \ \ \mbox{and}\ \ 
 \Psi(x) := \int_x^{\infty} \psi(y)\, \md y.
$$

We begin by showing \eqref{eq:Xuptail}.
Since $ |g_k| \leq g_* <\infty $,
with $ b_m $ defined in \eqref{eq:b-m}, 
we have that $b_m \to -a/2 $ as $m \to \infty$. 
This implies that $ \{b_m\}_{m \ge 1}$ is bounded, and hence there exists a constant $g_{**}$ such that 
\begin{equation}
	\label{eq:boundedness}
	\mbox{for all}\ \ m \ge 2,\ \ k = 1, \ldots, m^2,\ \ |g_k^{(m)}| \le g_{**}<\infty.
\end{equation}
Consequently,
$ X^{(m)}_k $ solves the equation \eqref{eq:Xm}
with drift coefficient being at most $ g_{**} $,
thereby
\begin{align}
	\label{eq:cmp}
	\MP\Big( \sup_{0\leq t \leq T} X^{(m)}_k(t) \geq u \Big)
	\leq
	\MP\Big( \sup_{0\leq t \leq T} (X^{(m)}_k(0)+W(t)+g_{**}t) \geq u \Big),
\end{align}
where $ W(t) $ is a standard Brownian motion.
Using the reflection principle 
$ \MP( \sup_{0\leq t \leq T} W(t) \geq a ) = 2 \Psi((\frac{a}{\sqrt{t}})_+) $
to bound the l.h.s.\ of \eqref{eq:cmp},
we further obtain
\begin{align*}
	\MP\Big( \sup_{0\leq t \leq T} X^{(m)}_k(t) \geq u \Big)
	\leq
	2 \ME \Psi\Big( \tfrac{u-g_{**}T-X^{(m)}_k(0)}{\sqrt{T}} \Big).
\end{align*}
Now, fix $ k\in\{1,\ldots, m\} $.
By \eqref{eq:lambda:1m} and \eqref{eq:a:cnd}, 
we have that $ \la^{(m)}_k = \lambda_k \geq c_*k \geq c_*  $,
where $ c_* := a + 2\inf_{n \ge 1}\ol{g}_n >0 $.
With this,
letting $ (\zeta_k)_{k=1}^\infty \sim \bigotimes_{k=1}^\infty \Exp(c_*) $,
we have $ X^{(m)}_k(0) \preceq \xi_k := \sum_{j=1}^k \zeta_j $.
Since $ x \mapsto \Psi( \tfrac{u-g_{**}T-x}{\sqrt{T}} ) $ is increasing,
by the preceding stochastic dominance we have
\begin{align}
	\label{eq:xi:bd}
	\MP\Big( \sup_{0\leq t \leq T} X^{(m)}_k(t) \geq u \Big)
	\leq
	2 \ME \Psi\big( \tfrac{1}{\sqrt T}(u-g_{**}T-\xi_k) \big).
\end{align}
For the Gaussian tail function $ \Psi(y) $ 
we have the following elementary inequality
\begin{align}
	\label{eq:Psi:bd}
	\Psi(y) \leq c e^{-(y_+)^2/2} \leq c e^{-cy\sqrt T/2},
\end{align}
where the second inequality follows 
from the fact that Gaussian tails decay faster than \emph{any} exponential tail.
Use this to further bound the r.h.s.\ of \eqref{eq:xi:bd}:
\begin{align*}
	\ME \Big( \Psi( \tfrac{1}{\sqrt T}(u-g_{**}T-\xi_k)  ) \Big) 
	\leq 
	c \ME \Big( e^{-c(u-\xi_k)/2} e^{g_{**}T/2} \Big)
	\leq
	c e^{-cu/2} \prod_{j=1}^k \ME(e^{c\zeta_j/2}),
\end{align*}
and combine this result with \eqref{eq:xi:bd}.
Recall the following elementary formula
\begin{align}
	\label{eq:exp-mgf}
	\ME\left(e^{v\zeta_j}\right) = \frac{c_*}{c_*-v}.
\end{align}
Further using this for $ v=c_*/2 $ (i.e.\ $ \ME(e^{c_*\zeta_j/2}) = 2$), 
we arrive at
\begin{align*}
	\MP\Big( \sup_{0\leq t \leq T} X^{(m)}_k(t) \geq u \Big)
	\leq
	c e^{-c_*u/2} \prod_{j=1}^k \ME(e^{c_*\zeta_j/2})
	=
	c e^{-cu/2} 2^{k}.
\end{align*}
This concludes the desired bound \eqref{eq:Xuptail}.

We now turn to the proof of \eqref{eq:Xlwtail}. 
Similarly to the preceding, here we have
\begin{align*}
	\MP\Big( \inf_{0\leq t \leq T} X^{(m)}_k(t) \leq u \Big)
	\leq
	2 \ME \Psi\Big( \tfrac{X^{(m)}_k(0) -g_{**}T-u}{\sqrt{T}} \Big).
\end{align*}
With $ \la^{(m)}_k $ defined in \eqref{eq:lambda:1m}--\eqref{eq:lambda:>m},
we clearly have that $ \la^{(m)}_k \leq \tilde{c}_* k  $,
for $ \tilde{c}_* := a + 2\sup_n \ol{g}_n <\infty $.
Consequently, 
letting $ (\tilde{\zeta}_k)_{k=1}^\infty \sim \bigotimes_{k=1}^\infty \Exp(\tilde{c}_* k) $,
we have $ X^{(m)}_k(0) \succeq \tilde{\xi}_k := \sum_{j=1}^k \tilde{\zeta}_j $,
and hence
\begin{align}
	\label{eq:xip:bd}
	\MP\Big( \inf_{0\leq t \leq T} X^{(m)}_k(t) \leq u \Big)
	\leq
	2 \ME \Psi\Big( \tfrac{\tilde{\xi}_k -g_{**}T-u}{\sqrt{T}} \Big).	
\end{align}
Fix $ k_* \geq 2/\tilde{c}^2_* $.
We consider the cases $ k \leq k_* $ and $ k > k_* $ separately.
For the former,
as $ x\mapsto \Psi(x) $ is decreasing and $ \tilde{\xi}_k > 0 $,
we bound the r.h.s.\ of \eqref{eq:xip:bd} by 
$ 2 \Psi( \frac{-g_{**}T-u}{\sqrt{T}} ) $.
By \eqref{eq:Psi:bd}, the last expression is bounded by $ c e^{cu} $, so
\begin{align*}
	\MP\Big( \inf_{0\leq t \leq T} X^{(m)}_k(t) \leq u \Big)
	\leq
	\frac{k^2_*}{k^2} e^{cu}
	=
	\frac{ce^{cu}}{k^2},
	\quad
	\text{for } k =1,\ldots,k_*.
\end{align*}
This concludes the desired inequality \eqref{eq:Xlwtail} for $ k\leq k_* $.

\medskip

The case $ k> k_* $ requires more refined estimates.
Fixing $ k\in\{k_*+1,\ldots,m^2\} $,
we begin by establishing a bound on the lower tail of $ \tilde{\xi}_k $.
To this end,
we consider the ``truncated'' variable 
\begin{align} 
	\tilde{\xi}'_k := \tilde{\xi}_k - \tilde{\xi}_{k_*} = \SL_{j=k_*+1}^k\tilde{\zeta}_j,
	\label{eq:xi-k}
\end{align}
together with the centered moment generating function 
\begin{align}
	f_k(v) := \ME\left(e^{v(\tilde{\xi}'_k-\ME(\tilde{\xi}'_k))}\right).
	\label{eq:mgf}
\end{align}
Recall that $\tilde\zeta_j \sim \Exp(\tilde c_*j)$, and $\ME \tilde\zeta_j = (\tilde c_*j)^{-1}$. 
With $ \tilde{\xi}'_k $ defined in \eqref{eq:xi-k}, 
using \eqref{eq:exp-mgf}, we calculate this function~\eqref{eq:mgf} explicitly as
\begin{align*}
	f_k(v) 
 	= 
 	\ME\exp\Bigl( v\SL_{j=k_*+1}^k\left(\tilde{\zeta}_j - \ME\tilde{\zeta}_j\right)\Bigr) 
 	= 
	\prod\limits_{j=k_*+1}^k\left( e^{-v\ME\tilde{\zeta}_j}\ME\left( e^{v\tilde{\zeta}_j}\right)\right)
	= 	
	\prod_{j = k_*+1}^k \frac{e^{-v/(\tilde{c}_*j)}}{1-v/(\tilde{c}_*j)},
\end{align*}
defined for all $ |v| < \tilde{c}_*(k_*+1) $.
We further express this as
\begin{align}
	\label{eq:mmgfxi}
	f_k(v) 
	= 
	\exp \Bigl( \sum_{j = k_*+1}^{k}
	\big( -\log\big(1-\tfrac{v}{\tilde{c}_*j} \big) - \tfrac{v}{\tilde{c}_*j} \big) \Bigr).
\end{align}
To bound the r.h.s.\ of \eqref{eq:mmgfxi},
apply Taylor's formula 
$ f(y) = f(0) + f'(0)y + \int_0^y (y-z)f'(z) dz $ with $ f(y) = \log(1+y)-y $ to obtain
\begin{align*}
	|\log(1+y)-y| 
		=
		\Big|\int_0^y \frac{z}{1+z} dz \Big|
	\leq
	k_* \Big|\int_0^y z dz\Big| \leq cy^2,
	\quad
	\text{for } |y| \leq \frac{k_*}{1+k_*} .
\end{align*}
Apply this inequality for $ y=v/(\tilde{c}_* j) $ in \eqref{eq:mmgfxi}, for $j = k_*+1, \ldots, k$. With $ \sum_{j=1}^\infty j^{-2} <\infty $,
we obtain $ f_k(v) \leq e^{c v^2} $ for $ |v| \leq \tilde{c}_*k_* $.
Combine the result with the Chernov bound
to obtain $ \MP( |\tilde{\xi}'_k-\ME(\tilde{\xi}'_k)| \geq x ) \leq e^{-xv+cv^2} $,
and substitute in $ v=\tilde{c}_*k_* $. 
We arrive at
\begin{align}
	\label{eq:tilxi'}
	\MP( |\tilde{\xi}'_k-\ME(\tilde{\xi}'_k)| \geq x ) 
	\leq 
	e^{ -\tilde{c}_*k_* x } e^{ c(\tilde{c}_*k_*)^2 }
	\leq
	c e^{ -\tilde{c}_*k_* x }.
\end{align}
This yields a tail bound on the variable $ \tilde{\xi}'_k $.
To relate the bound back to a lower tail bound on $ \tilde{\xi}_k  $,
we use $ \tilde{\xi}_k \geq \tilde{\xi}'_k $, followed by using \eqref{eq:tilxi'}, 
whereby obtaining
\begin{align*}
	\MP( \tilde{\xi}_k  \leq x ) 
	\leq 
	\MP( \tilde{\xi}'_k  \leq x ) 
	\leq	
	c e^{ -\tilde{c}_*k_* (\ME(\tilde{\xi}'_k)-x) }.	
\end{align*}
Further, as $ \ME(\tilde{\xi}'_k) $ and $ \ME(\tilde{\xi}_k) $
differ by $ \ME(\tilde{\xi}_{k_*}) \leq c $, we conclude
\begin{align}
	\label{eq:Yntail}
	\MP( \tilde{\xi}_k  \leq x )
	\leq
	c e^{ -\tilde{c}_*k_* (\ME(\tilde{\xi}_k)-x) }.	
\end{align}

Going back to proving \eqref{eq:Xlwtail},
we let $ F_k(x) := \MP(\tilde{\xi}_k \leq x) $ and $ G_k(x) := 1-F_k(x) $ 
denote the cumulative distribution function and the tail distribution function of $ \tilde{\xi}_k $, respectively.
Let $ \mu_k := \ME(\tilde{\xi}_k) $ denote the expected value.
By \eqref{eq:xip:bd} we have
\begin{align*}
	\MP\Big( \inf_{0\leq t \leq T} X^{(m)}_k(t) \leq u \Big)
	&\leq
	\ME \Psi\Big( \tfrac{\tilde{\xi}_k -g_{**}T-u}{\sqrt{T}} \Big)
	=
	\int_{\BR} \Psi(\tfrac{x -g_{**}T-u}{\sqrt{T}}) \md F_k(x)
\\
	&=
	-\int_{[\mu_k,\infty)} \Psi(\tfrac{x -g_{**}T-u}{\sqrt{T}}) \md G_k(x)
	+\int_{(-\infty,\mu_k]} \Psi(\tfrac{x -g_{**}T-u}{\sqrt{T}}) \md F_k(x).
\end{align*}
Apply integration by parts
(with $ f(x) = \Psi(\tfrac{x -g_{**}T-u}{\sqrt{T}}) $). Using $G_k(\infty) = F_k(-\infty) = 0$, we get:
\begin{align}
	\label{eq:IPS1}
	-\int_{[\mu_k,\infty)} f(x) \md G_k(x)
	&=
	f(\mu_k) G_k(\mu_k)
	+	
	\int_{[\mu_k,\infty)} f'(x) G_k(x) \md x,
\\	
	\label{eq:IPS2}
	\int_{(-\infty, \mu_k)} f(x) \md F_k(x)
	&=
	f(\mu_k) F_k(\mu_k)
	-
	\int_{(-\infty,\mu_k)} f'(x) F_k(x) \md x.
\end{align}
With $ f'(x) = - \frac{1}{\sqrt{T}}\psi(\frac{x-g_{**}T-u}{\sqrt{T}})< 0 $,
we drop the last term in \eqref{eq:IPS1}.
Further using $ F_k(\mu_k)+G_k(\mu_k)=1 $, summing \eqref{eq:IPS1}--\eqref{eq:IPS2}, we obtain:
\begin{align}
	\label{eq:Psi:bd:}
	\MP\Big( \inf_{0\leq t \leq T} X^{(m)}_k(t) \leq u \Big)
	\leq
	\Psi(\tfrac{\mu_k -g_{**}T-u}{\sqrt{T}})
	+
	\int_{(-\infty,\mu_k]} \tfrac{1}{\sqrt{T}}\psi(\tfrac{x-g_{**}T-u}{\sqrt{T}}) F_k(x) \md x.
\end{align}
Next, to further bound the r.h.s.\ of \eqref{eq:Psi:bd:},
we first note that
\begin{align}
	\label{eq:muk}
	\mu_k = \tilde{c}_* \sum_{j=1}^k j^{-1} \geq \tilde{c}_* \log k -c.
\end{align}
Using this and \eqref{eq:Psi:bd},
we bound the first term on the r.h.s.\ of \eqref{eq:Psi:bd:} as
\begin{align}
	\label{eq:Psi:bd:1st}
	\Psi(\tfrac{\mu_k -g_{**}T-u}{\sqrt{T}}) \leq ce^{-c(\log k - c - u)_+^2}
	\leq
	ce^{-\frac{c}{2}(\log k - u)_+^2}.
\end{align}
(Here we put $ \frac{c}{2} $ just to clarify that the second inequality
follows by making the constant in the exponential smaller.)
As for the integral term in \eqref{eq:Psi:bd:}, 
we use the tail estimate \eqref{eq:Yntail} to bound $ F_k(x) $ by $ ce^{-\tilde{c}_*k_*(\mu_k-x)} $,
and replace the integral over $ (-\infty,\mu_k] $ by an integral over the entire $ \BR $,
followed by the change of variable $ \frac{\mu_k-x}{\sqrt T}\mapsto x $.
This yields
\begin{align*}
	\int_{(-\infty,\mu_k]} \tfrac{1}{\sqrt{T}}\psi(\tfrac{x-g_{**}T-u}{\sqrt{T}}) F_k(x) \md x
	\leq
	c\int_{\BR} \psi\big(\tfrac{\mu_k-g_{**}T-u}{\sqrt{T}}-x \big) e^{-\tilde{c}_*k_*\sqrt{T}x} \md x.
\end{align*}
The last integral is explicitly evaluated to be
$ 
	e^{(\tilde{c}_*k_*\sqrt{T})^2/2} e^{-\tilde{c}_*k_*(\mu_k-g_{**}T-u)} 
	\leq
	c e^{-\tilde{c}_* k_*(\mu_k-u)}.
$
Combining this with the estimate of $ \mu_k $ \eqref{eq:muk}, 
followed by using $ \tilde{c}^2_* k_* \geq 2 $,
we conclude
\begin{align}
	\label{eq:Psi:bd:2st}
	\int_{(-\infty,\mu_k]} \tfrac{1}{\sqrt{T}}\psi(\tfrac{x-g_{**}T-u}{\sqrt{T}}) F_k(x) \md x
	\leq
	ce^{-\tilde{c}^2_* k_*\log k +\tilde{c}_* k_* u}
	\leq
	ck^{-2} e^{cu}.
\end{align}
Inserting \eqref{eq:Psi:bd:1st}--\eqref{eq:Psi:bd:2st} into \eqref{eq:Psi:bd:},
we conclude the desired estimate \eqref{eq:Xlwtail} for $ k > k_* $.
\end{proof}

{\it{}Step 3.} We now return to showing the convergence of $ [X^{(m)}]_{\downarrow n} $.
For $ f \in C([0,T],\BR^n) $, we let 
\begin{align*}
	\text{osc}_\delta(f):= 
	\sup \big\{ \Vert f(t)-f(s) \Vert_2 \ \big| \ s,t\in[0,T], |t-s|\leq\delta \big\}
\end{align*}
denote the modulus of continuity of $f$, measured in the Euclidean norm $ \Vert x \Vert_2 := \sqrt{x_1^2+\ldots+x_n^2} $.
Since $ X^{(m)} = (X^{(m)}_i)_{i=1}^{m^2}$ solves the equation \eqref{eq:Xm} 
with drift coefficients bounded by $ g_{**} $,
we have
\begin{align}
	\label{eq:osc}
	\sup_{m \geq 1}
	\MP\Big( \text{osc}_{\delta} \big( \, [X^{(m)}]_{\downarrow n} \, \big) \geq \eps \Big)
	\to 0,
	\text{ as } \delta \to 0,
\end{align}
for any fixed $ \eps>0 $.
With this, by Arzel\`{a}--Ascoli theorem, 
it is standard to show that $ \{ [X^{(m)}]_{\downarrow n} \}_{m \geq 1} $
is tight in $ C([0,T],\BR^n) $.
By the Skorohod representation theorem,
after passing to a subsequence and a different probability space,
we have
\begin{align}
	[X^{(m)}_k]_{\downarrow n} \to (X_k)_{k=1}^n  \text{ in }  C([0,T],\BR^n),
	\text{ as } m\to\infty,
	\quad
	\text{a.s.}
\label{eq:X-named-conv}
\end{align}
The limit process $ X  = (X_i)_{i \ge 1}$ 
is taken to be independent of $ T$ and $ n $
by a standard diagonal argument.

\medskip

{\it{}Step 4.} 
We now proceed to show that $ X $ has gap distribution $ \pi_a $. 
Fix $T > 0$ and $n = 1, 2, \ldots$. 
We do this by first establishing 
the convergence of $ [Y^{(m)}]_{\downarrow n} $.
By \eqref{eq:Xuptail} we have
\begin{align*}
	\MP\Big( \max_{k=1,\ldots,n} \sup_{t\leq T} X^{(m)}_k(t) > u \Big) 
	\leq
	\sum_{k=1}^n \MP\Big( \sup_{t\leq T} X^{(m)}_k(t) > u \Big) 
	\leq
	cne^{c(cn-u)}\to 0,
	\quad
	\text{ as } u \to \infty.
\end{align*}
Fix an arbitrarily small $ \eps >0 $.
With $ n,T $ already being fixed,
we now choose a sufficiently large $ u\in\BR_+ $ such that
\begin{align}
	\label{eq:rk:trnc}
	\MP\big( X^{(m)}_k(t) \leq u, \ \forall k=1,\ldots,n, \ t \leq T \big) \geq 1 -\eps/2,
	\quad
	\text{ for all } m \geq n.
\end{align}
That is, with probability at least $ 1-\eps/2 $, 
the first $ n $ named particles $ X^{(m)}_1,\ldots, X^{(m)}_n $
always stay below the level $ u $ within the time interval $ [0,T] $.
With this $ u $, we further apply \eqref{eq:Xlwtail} to obtain
\begin{align}
	\notag
	\MP\Big( \min_{n' \leq k \leq m^2}  \inf_{0\leq t\leq T} X^{(m)}_k(t) \leq u \Big) 
	&\leq
	\sum_{k =n'}^{m^2} \MP\Big( \inf_{0\leq t\leq T} X^{(m)}_k(t) \leq u \Big) 
\\
	\label{eq:rk:trnc::}
	&\leq 
	c \sum_{k =n'}^\infty (e^{-c(\log k - u)_+^2}+{k^{-2}e^{cu}}).
\end{align}
Since $ \sum_{k =1}^\infty (e^{-c(\log k - u)_+^2}+k^{-2}e^{cu}) <\infty $,
the last expression in \eqref{eq:rk:trnc::} clearly tends to zero as $ n'\to\infty $. 
Fix some $ \tilde{n} \geq n $ such that
\begin{align}
	\label{eq:rk:trnc:}
	\MP\big( X^{(m)}_k(t) \geq u, \ \forall k > \tilde{n}, \ t \leq T \big) \geq 1 -\eps/2.
\end{align}
That is, with probability at least $ 1-\eps/2 $,
none of the name particles $ X_{\tilde{n}+1}, X_{\tilde{n}+2}, \ldots $
ever reaches a level below $ u $ within $ [0,T] $.
Let $ \mathcal{R} : \BR^{\tilde{n}} \to \BR^{\tilde{n}} $,
$ \mathcal{R}(x) := (x_{\mP_x(i)})_{i=1}^{\tilde{n}} $,
denote the ranking map of an $ \tilde{n} $-tuple.
By \eqref{eq:rk:trnc} and \eqref{eq:rk:trnc:},
we have that
\begin{align*}
	\MP \Big( 
		\big[ \, \mathcal{R} \big( [X^{(m)}(t)]_{{\downarrow \tilde{n}}} \big) \, \big]_{\downarrow n} 
		= [Y^{(m)}(t)]_{\downarrow n},\ \forall t\leq T 
	\Big)
	\geq 1 -\eps.
\end{align*}
Namely, with probability at least $ 1-\eps $, to obtain the first $ n $ ranked particles 
$ Y^{(m)}_1(t), \ldots , Y^{(m)}_n(t) $ within the system of $m^2$ particles 
$X_1^{(m)}, \ldots, X_{m^2}^{(m)}$, 
it suffices to rank \emph{only} $ X^{(m)}_1(t),\ldots, X^{(m)}_{\tilde{n}}(t) $ (as opposed to ranking all $m^2$ named particles), 
and take the first $ n $ resulting particles.
Because the map $ \mathcal{R} $ is continuous, and 
$ [X^{(m)}]_{\downarrow \tilde{n}} \to [X]_{\downarrow \tilde{n}} $ in $ C([0,T],\BR^{\tilde{n}}) $
(from~\eqref{eq:X-named-conv}),
it follows that 
\begin{align}
	\label{eq:Ym:cnvg}
	[Y^{(m)}]_{\downarrow n} \to [Y]_{\downarrow n}
	\text{ in } C([0,T],\BR^{n}) \text{ a.s., } 
	\
	\text{where}
	\
	Y_k(t) := X_{\mP_{X(t)}(k)}(t).
\end{align}
Further, since we have $ (Z^{(m)}_k(t))_{k=1}^m \sim \pi^{(m)}_a $,
letting $ m\to\infty $ we further obtain that
\begin{align*}
	(Z_k(t))_{k=1}^\infty \sim \bigotimes_{k=1}^\infty \Exp(\lambda_k),
	\ 
	\forall t \in \BR_+,
	\quad
	\text{where } Z_k(t) := Y_{k+1}(t) - Y_k(t).
\end{align*}
We have thus concluded that
the gap process $Z(t)$ of the system $ X $ is distributed according to $ \pi_a $ 
for all $ t\in\BR_+ $.

\medskip

{\it{}Step 5.} Finally, we still need to show that $ X $ solves \eqref{eq:infinite-CBP}.
Doing so requires first showing that $ Y $ solves the corresponding equation 
\eqref{eq:infinite-CBP:rk}.
Indeed, the ranked process $Y^{(m)}$ solves the following finite system of SDEs:
\begin{equation}
	\label{eq:Y-k}
	Y_k^{(m)}(t) 
	= 
	Y^{(m)}_k(0) + g_k^{(m)}t + B^{(m)}_k(t) 
	+ \tfrac12L_{(k-1, k)}^{(m)}(t) - \tfrac12L_{(k, k+1)}^{(m)}(t),\ 
	\ k = 1, \ldots, m^2,
\end{equation}
where the local time $L^{(m)}_{(k, k+1)}$ and
the Brownian motion $ B^{(m)}_k $ are defined in Section~\ref{sect:finite}.
Note that although we take $ (W_k)_{k=1}^\infty $ to be fixed (i.e.\ independent of $ m $),
the Brownian motions $ B^{(m)}_k $ (defined in \eqref{eq:Bk})
still depend on $ m $. 
However, with $ [B^{(m)}]_{\downarrow n} $ being tight in $ C([0,T],\BR^n) $,
applying again the Skorohod representation theorem
(after passing to a finer subsequence and yet another probability space),
we assume without lost of generality $ [B^{(m)}]_{\downarrow n} \to [B]_{\downarrow n} $ in $ C([0,T],\BR^n) $,
where $ B(t)=(B_k(t))_{k=1}^\infty $, 
and $ B_k(t) $, $ k=1,2,\ldots $ are independent standard Brownian motions.

Now, as we already have that $ Y^{(m)} $ and $ B^{(m)} $ converge,
taking $ m\to\infty $ in \eqref{eq:Y-k} for $ k=1 $ yields
\begin{align*}
	\tfrac12L_{(1, 2)}^{(m)}
	\to
	\tfrac12L_{(1, 2)}
	\text{ in }
	C[0,T] \text{ a.s.},
	\
	\text{ where }
 \tfrac12L_{(1, 2)}(t)
	:=
	- Y_1(t) + Y_1(0) + g_1t + B_1(t).
\end{align*}
Performing this $ m\to\infty $ procedure inductively for $ k=2,3,\ldots $,
we further obtain
\begin{align*}
	&
	\tfrac12L_{(k, k+1)}^{(m)}
	\to
	\tfrac12L_{(k, k+1)}
	\text{ in }
	C[0,T] \text{ a.s.},
\end{align*}
where $ L_{(k,k+1)}(t) $ is defined inductively through the following relation
\begin{align*}
 \tfrac12L_{(k, k+1)}(t) :=
\tfrac12L_{(k-1, k)}(t)
	- Y_k(t) + Y_k(0) + g_kt + B_k(t).
\end{align*}
Next, each $L_{(k, k+1)}$ is continuous, nondecreasing, and starts from zero: $L_{(k, k+1)}(0) = 0$. 
This is so because each $ L^{(m)}_{(k,k+1)} $ has all these properties, 
and they are preserved in limits under the uniform topology of $ C[0,T] $. 
Furthermore, $L_{(k, k+1)}$ increases only when $Y_k = Y_{k+1}$. 
To see this,
we consider a generic $ t\in[0,T] $
such that $ Y_k(t) < Y_{k+1}(t) $.
By the continuity of $ Y_k(t) $ and $ Y_{k+1}(t) $,
we must also have $ Y_k(s) < Y_{k+1}(s) $ for $ s \in [t', t''] $,
for some $ t' <t''\in[0,T] $.
With this, for all large enough $ m $, 
we have $ Y^{(m)}_k(s) < Y^{(m)}_{k+1}(s)$, $ s \in [t', t'']$. 
From the properties of the local time for finite systems, we get: $ L_{(k, k+1)}^{(m)}(t') = L_{(k, k+1)}^{(m)}(t'') $.
Letting $ m \to \infty $ yields $ L_{(k, k+1)}(t') = L_{(k, k+1)}(t'') $, 
which proves that $L_{(k, k+1)}$ increases only when $ Y_k = Y_{k+1} $. 
With the aforementioned properties of $ L_{(k,k+1)} $,
we conclude that $ L_{(k,k+1)} $ is the local time 
of collision between $ Y_{k} $ and $ Y_{k+1} $,
and hence that $ Y $ solves \eqref{eq:infinite-CBP:rk}.

We now return to proving that $ X $ solves \eqref{eq:infinite-CBP}.
This is done by taking the $ m\to\infty $ limit
of the finite system of equations \eqref{eq:Xm}
similarly to the way we did it for $ Y^{(m)} $.
However, unlike \eqref{eq:Y-k},
the diffusion coefficients in \eqref{eq:Xm} 
are generally discontinuous due the exchange of ranks. 
We resolve this problem following \cite{MyOwn6}, by first showing:

\begin{lemma}
Define the following random set: 
$$
	\CN := \{t \in [0, T]\mid Y(t) \text{ has a tie}\,\} 
	= \{t \in [0, T]\mid\, \exists k \ge 1:\, Y_k(t) = Y_{k+1}(t)\}.
$$
Then $\MP$-a.s.\ the Lebesgue measure of $ \CN $ is equal to zero.
\label{lemma:N-zero}
\end{lemma}

\begin{proof} 
As $ Y $ solves \eqref{eq:infinite-CBP:rk} (as proven above),
the desired result follows once we show that 
the infinite system \eqref{eq:infinite-CBP:rk} can be reduced
to a \emph{finite system} at any given level $ u\in\BR $.
More precisely, fixing arbitrary $ u\in\BR $ and $ T\in\BR_+ $,
we aim at showing
\begin{align}
	\label{eq:trunc:}
	\inf_{0 \le t \le T}Y_k(t) < u,
	\
	\text{ for only finitely many } k,
	\quad
	\text{a.s.}
\end{align}
Once this is established,
the desired result follows by the same argument in \cite[Lemma 3.9]{MyOwn6}.
Turning to showing \eqref{eq:trunc:},
because $Y_k(t) \le Y_{k+1}(t)$ a.s., we need only to show that 
\begin{equation}
	\label{eq:05}
	\MP\left(\inf_{0 \le t \le T}Y_k(t) < u\right) \to 0,\ \ k \to \infty.
\end{equation}
As $ Y^{(m_)}_k \to Y_k $ in $ C[0, T] $,
and the set $ \{y(\Cdot)\,|\, \inf_{0 \le t \le T} y(t) < u \} $ is open in $ C[0, T] $,
we have
\begin{equation}
	\label{eq:06}
	\MP\left(\inf_{0 \le t \le T}Y_k(t) < u\right) 
	\le 
	\varliminf_{m \to \infty}
	\MP\left(\inf_{0 \le t \le T}Y^{(m)}_k(t) < u\right).
\end{equation}
Next, since $ Y^{(m)}_k(t) $ is the $ k $th ranked particle in $ X^{(m)}(t) $,
it follows that
\begin{align*}
	Y^{(m)}_k(t) \geq \min_{j=k,\ldots,m^2} X^{(m)}_j(t),\ \ 
\mbox{so}\ \ 
	\MP\Big( \inf_{0 \le t \le T} Y^{(m)}_k(t) < u \Big) 
	\leq 
	\sum_{j=k}^{m^2} \MP\Big( \inf_{0 \le t \le T} X^{(m)}_k(t) < u \Big).
\end{align*}
Now, applying \eqref{eq:Xlwtail} to bound the r.h.s., 
we arrive at
\begin{align}
	\label{eq:Y:trun}
	\MP\Big( \inf_{0 \le t \le T} Y^{(m)}_k(t) < u \Big) 
	\leq 
	c\sum_{j=k}^{m^2} (e^{-c(\log j-u)_+^2}+ j^{-2}e^{cu} )
	\leq 
	c\sum_{j=k}^{\infty} (e^{-c(\log j-u)_+^2}+ j^{-2}e^{cu} ).
\end{align}
The last term in \eqref{eq:Y:trun} is independent of $ m $ and tends to zero as $ k\to\infty $
(as explained previously after \eqref{eq:rk:trnc::}).
Consequently,
inserting \eqref{eq:Y:trun} into \eqref{eq:06} yields the desired result \eqref{eq:05}.
\end{proof}
\noindent
With this Lemma~\ref{lemma:N-zero}, the rest of the proof 
follows by the same argument to the end of the proof of \cite[Theorem 3.3]{MyOwn6}, 
starting from \cite[Lemma 6.5]{MyOwn6} up to the end of the proof of this theorem.
That is, as $ m\to\infty $, the solution $ X^{(m)}_k $
of the finite system \eqref{eq:finite-CBP}
converges to a solution of the infinite system \eqref{eq:infinite-Atlas}.

\subsection{Proof of Part~(b)}
Fix $ t\in\BR_+ $ and an integer $ k $. 
It follows from Proposition~\ref{prop:stability} that 
\begin{align}
	\label{eq:Ym:exp}
	\ME\big( Y^{(m)}_k(t) - Y_k^{(m)}(0) \big) = -\tfrac{a}{2}t ,
\end{align}
because $\ol{g}^{(m)} = -a/2$ for all $m$.
Our goal is to pass \eqref{eq:Ym:exp} to the limit $ m\to\infty $.
To this end,
since we already have 
$Y_k^{(m)}(s) \to Y_k(s) $ a.s.\ for $ s=t $ and $ s=0 $,
it suffices to establish the $ L^2 $-boundedness of $ \{ Y^{(m)}_k(t)\}_{m \geq 1} $:
\begin{equation}
	\label{eq:13}
	\sup_{m\geq 1} \ME\big( Y^{(m)}_k(t) \big)^2 <\infty.
\end{equation}
Indeed, \eqref{eq:13} guarantees the uniform integrability of $ \{ Y^{(m)}_k(t)\}_{m \geq 1} $,
so almost sure convergence implies convergence of expectations.
Further, since
\begin{align*}
	Y^{(m)}_k(t) = Y^{(m)}_1(t) + Z_{k-1}^{(m)}(t) + \ldots + Z_1^{(m)}(t)
\end{align*}
and $ Z_{k}^{(m)}(t) \sim \Exp(\lambda_k) $ for $ k \leq m $,
\eqref{eq:Ym:exp} clearly follows once we prove it for $ k=1 $.

Proceeding to prove \eqref{eq:Ym:exp} for $ k=1 $,
we recall that $ \preceq $ denotes stochastic dominance.
Apply comparison techniques from \cite[Corollary 3.7]{MyOwn2} to obtain that, 
for all $m \ge 1$, 
$ Y^{(m)}_1(t) \preceq g_1t + W(t) $,
where $ W $ is some standard Brownian motion.
From this it follows that 
\begin{align}
	\label{eq:Yup}
	\sup_{m \geq 1} \ME\big( Y^{(m)}_k(t)_+ \big)^2 \le \ME\big( (g_1t + W(t))_+\big)^2 <\infty.
\end{align}
Next, to bound $ \ME(Y^{(m)}_k(t)_-)^2 $,
we fix $ u \in \BR_+ $ and write
\begin{align}
	\label{eq:Y1lw}
	\MP( Y^{(m)}_1(t) \leq -u )
	=
	\MP\Big( X^{(m)}_k(t) \leq -u \Big)
	\leq
	\sum_{ k = 1 }^{\infty} \MP\Big( \inf_{k \geq 1} X^{(m)}_k(t) \leq -u \Big).
\end{align}
Combining this with \eqref{eq:Xlwtail}, followed by using
$ e^{-c(u+\log k)^2} \leq e^{-cu^2-c(\log k)^2} $,
we obtain
\begin{align}
	\label{eq:last}
	\MP( Y^{(m)}_1(t) \leq -u )
	\leq&
	c \sum_{ k = 1 }^{\infty} (e^{-cu^2-c(\log k)^2}+k^{-2}e^{-cu})\leq
	cS_1e^{-cu^2} + cS_2e^{-cu},
\\
	\label{eq:two-series}
	&\text{where } S_1 := \SL_{k=1}^{\infty}e^{-c(\log k)^2},\ \ S_2 := \SL_{k=1}^{\infty}k^{-2}.
\end{align}
As both these series in~\eqref{eq:two-series} converge, 
\eqref{eq:last} shows that $ Y^{(m)}_1(t) $ has an exponential lower tail which is uniformly in $ m $,
so in particular $ \sup_{m \geq 1} \ME(Y^{(m)}_1(t)_-)^2 <\infty $.
Combining this with \eqref{eq:Yup} yields the desired result \eqref{eq:13}.

\appendix
\section{}
\label{sect:appd}

Here we provide bounds on the number of particles under the measure $ \pi_a $.
\begin{prop}
\label{prop:singular}
Fix $ g_1,g_2,\ldots $ satisfying the condition~\eqref{eq:drifts-stable}
and fix $ a\in\BR $ satisfying~\eqref{eq:a:cnd}. Let 
\begin{align*}
	0 = \xi_1 < \xi_2 < \xi_3 < \ldots \in \BR_+ 
\end{align*}
be a random configuration of points with the gap distribution
$ (\zeta_k:=\xi_{k+1}-\xi_{k})_{k=1}^\infty \sim \pi_a $.
Let 
$	N(x) := \#\left\{i \ge 1\mid \xi_i \le x\right\} $
denote the random number of $ \xi_k $-particles on the interval $ [0, x] $.
Then 
\begin{align*}
	0 < \liminf_{x\to\infty} e^{-ax} N(x)
	\leq
	\limsup_{x\to\infty} e^{-ax} N(x) < \infty\ \ \mbox{a.s.}
\end{align*}
\end{prop}
\begin{proof} 
Let $ \lambda_k := 2(g_1+\ldots+g_k) + ka $.
Under the conditions~\eqref{eq:drifts-stable} and \eqref{eq:a:cnd},
we have
\begin{align}
	\label{eq:zeta:av}
	\sum_{k=1}^n \ME \zeta_k 
	= 
	\sum_{k=1}^n \frac1{\la_k} 
	= a^{-1} \log n + c_n,
\end{align}
where $(c_k)_{k \ge 1}$ is a bounded deterministic sequence,
and $ \sum_{k=1}^\infty \Var( \zeta_k) = \sum_{k=1}^n \frac1{\la^2_k} < \infty $.
With the last condition,
\cite[Theorem 1.4.2]{StroockBook} implies that the series 
$ \sum_{k=1}^{\infty}(\zeta_k - \ME \zeta_k) $ converges a.s.
Combining this with \eqref{eq:zeta:av} yields
\begin{align*}
	 \sup_{n\geq 1}\Bigl|\xi_n - a^{-1}\log n  \Big| <\infty
	\quad
	\text{a.s.,}
\end{align*}
which clearly implies the desired result. 
\end{proof}

We next show that the measures $ \pi_a $ are all mutually singular for different values of $a$. 
\begin{prop}
\label{prop:dens}
Fixing $ g_1,g_2,\ldots $ satisfying the condition~\eqref{eq:drifts} and 
$ a> a' >-2\inf_n \bar{g}_n  $,
we have that the measures $ \pi_a $ and $ \pi_{a'} $ are mutually singular.
\end{prop}
\begin{proof}
Under the measure $ \pi_a $, we have that $ \frac{1}{2n\ol{g}_n+na}Z_n $, $ n=1,2,\ldots $,
are i.i.d.\ $ \Exp(1) $ variables.
For $ Z\sim\Exp(1) $, the variable $ U:=\ME(\log Z) $ is integrable (i.e. $ \ME|U|<\infty $),
so, letting $ \mu := \ME(U) $,
by the strong Law of Large Numbers we that
\begin{align}
	\label{eq:singular:a}
	\frac{1}{n} \sum_{k=1}^n \log\Big(\frac{Z_k}{2k\ol{g}_k+ka}\Big) \to \mu,
	\quad
	\pi_a\text{-a.s.},
\\
	\label{eq:singular:a'}
	\frac{1}{n} \sum_{k=1}^n \log\Big(\frac{Z_k}{2k\ol{g}_k+ka'}\Big) \to \mu,
	\quad
	\pi_{a'}\text{-a.s.}
\end{align}
Under the conditions~\eqref{eq:drifts} and $ a,a'>-2\inf_n g_n $, it is straightforward to show that
\begin{align*}
	\frac{1}{n} \sum_{k=1}^n \log\Big(\frac{2k\ol{g}_k+ka'}{2k\ol{g}_k+ka}\Big) \to \log\Big(\frac{a'}{a}\Big).
\end{align*}
Adding this to \eqref{eq:singular:a'} yields
\begin{align*}
	\frac{1}{n} \sum_{k=1}^n \log\Big(\frac{Z_k}{2k\ol{g}_k+ka}\Big) \to \mu + \log\Big(\frac{a'}{a}\Big),
	\quad
	\pi_{a'}\text{-a.s.}
\end{align*}
This, with $ a\neq a' $, concludes that $ \pi_a $ and $ \pi_{a'} $ are mutually singular.
\end{proof}

\end{document}